\documentclass{amsart}
\usepackage{amssymb,latexsym, amsmath, amscd, array, graphicx}

\swapnumbers
\numberwithin{equation}{section}

\theoremstyle{plain}

\newtheorem{thm}{Theorem}[section]
\newtheorem{cor}[thm]{Corollary}

\newtheorem{question}[thm]{Problem}
\newtheorem{prop}[thm]{Proposition}

\theoremstyle{definition}

\newtheorem{remark[thm]}{Remark}

\def\Hom{\protect\operatorname{Hom}}

\def\cd{\protect\operatorname{cd}}

\def\Z{{\mathbb Z}}
\def\Q{{\mathbb Q}}
\def\R{{\mathbb R}}
\def\N{{\mathbb N}}
\def\k{{\bold k}}

\def\1{\hbox{\rm\rlap {1}\hskip.03in{\rom I}}}
\def\Bbbone{{\rm1\mathchoice{\kern-0.25em}{\kern-0.25em}
{\kern-0.2em}{\kern-0.2em}I}}

\long\def\forget#1\forgotten{} %

\begin{document}

\title[On dimension of product of groups ]
{On  dimension of product of groups}
\author[A.~Dranishnikov]
{Alexander Dranishnikov}
\address{A. Dranishnikov, Department of Mathematics, University
of Florida, 358 Little Hall, Gainesville, FL 32611-8105, USA}
\email{dranish@math.ufl.edu}

\begin{abstract} 
We prove the product formula 
$$
\cd(\Gamma\times\Gamma)=2\cd\Gamma$$ for cohomological dimension of geometrically finite groups.
\end{abstract}

\maketitle

\section{Introduction}
We recall that the cohomological dimension of a group $\Gamma$ is defined as $$\cd\Gamma=\max\{n\mid H^n(\Gamma,M)\ne 0\}$$ where the maximum is taken over all
$\Z\Gamma$-modules $M$. If there is no such maximum we write $\cd\Gamma=\infty$. 
By the Eilenberg-Ganea theorem~\cite{Br} the cohomological dimension $\cd\Gamma$ coincides with the geometric dimension of $\Gamma$ whenever $\cd\Gamma\ne 2$. The geometric dimension of  $\Gamma$
is the minimal dimension of CW-complexes representing the Eilenberg-MacLane  space $K(\Gamma,1)$.

Bestvina and Mess made a connection between the cohomological dimension of a hyperbolic torsion free group and topological dimension of its boundary~\cite{BM},\cite{Be}:
$$
\cd\Gamma=\dim\partial\Gamma+1.
$$
In the 30s Pontryagin constructed compact subsets $X_p\subset\R^5$ for all prime $p$, now  called Pontryagin surfaces, such that $\dim(X_p\times X_q)=3$ whenever $p\ne q$
and $\dim(X_p\times X_p)=4$ for all $p$. 
Then in the late 40s Boltyanskii constructed a 2-dimensional compact subset $B\subset\R^5$ with $\dim(B\times B)=3$. In~\cite{Dr1} the right-angled Coxeter groups with boundaries
Pontryagin surfaces $X_p$ were  constructed  for all
$p$. This brought an example of a family of hyperbolic groups $\Gamma_p$ with $\cd\Gamma_p=3$ for all $p$ and $\cd(\Gamma_p\times\Gamma_q)=5$ for $p\ne q$~\cite{Dr2}.
Also in~\cite{Dr1} it was proven that Boltyanskii compactum cannot be a boundary of a Coxeter group. This result left some hope for
the logarithmic law $$\cd(\Gamma\times\Gamma)=2\cd\Gamma$$ for all geometrically finite groups. The above equality is the main result of this paper.
We recall that a  group $\Gamma$ is called geometrically finite if
there is a finite Eilenberg-MacLane  complex $K(\Gamma,1)$.
Generaly, for groups which are not geometrically finite the above equality does not hold: For example, for  the group of rationals $\cd\Q=2$ and $\cd(\Q\times\Q)=3$.

The author is thankful to Alex Margolis for valuable remarks.

\section{Dimension of the square of a group}

\subsection{Cohomological dimension of group with respect to a ring} Let $R$ be a commutative ring with unit and let $\Gamma$ be a group. By $R\Gamma$ we denote the group ring,
Recall that $R\Gamma$ is the ring of all functions $f:\Gamma\to R$ with finite support and the convolution as the product.
The cohomological dimension of a group $\Gamma$ with respect to the ring
$R$ is defined as follows~\cite{Br},\cite{Bi},
 $$\cd_R\Gamma=\max_M\{n\mid H^n(\Gamma,M)\ne 0\}$$ where the maximum is taken over all $R\Gamma$-modules. In the case  $R=\Z$ we use the notation
$\cd\Gamma=\cd_\Z\Gamma$
and call the number $\cd\Gamma$ the {\em  cohmological dimension} of $\Gamma$.

For a $\Gamma$-module $M$ we use notations $H^*(\Gamma, M)$ for cohomology groups of a discrete group $\Gamma$  and $H^*(X;M)$ for 
cohomology groups of a topological space $X$ with the fundamental group $\pi_1(X)=\Gamma$. Note that a $\Gamma$-module defines a locally trivial sheaf $\mathcal M$ on $X$ and
$H^*(X;M)$ can be treated as cohomology with coefficients in the sheaf~\cite{Bre}, $H^*(X;M)=H^*(X;\mathcal M)$. Note that $H^*(\Gamma,M)=H^*(B\Gamma;M)$.

For abelian groups $A$ and $B$ we are using the notation $A\ast B=Tor(A,B)$. A group $\Gamma$ is a group of finite type if there is a classifying CW complex $B\Gamma$ with finite skeletons $B\Gamma^{(n)}$ for all $n$.

We recall the Cohomology Universal Coefficient Formula for twisted coefficients

{\bf UCF:}
{\em For a group $\Gamma$ of finite type, a $\Z\Gamma$-module $M$, and an abelian group $R$ there the Universal Coefficient Formula}~\cite{Bre}:
$$
H^n(\Gamma,M\otimes R)=H^n(\Gamma,M)\otimes  R\ \bigoplus\  H^{n+1}(\Gamma,M)\ast R.
$$
Note that $R\Gamma=\Z\Gamma\otimes R$ for any commutative ring $R$. Then the UCF turns into
$$
H^n(\Gamma,R\Gamma)=H^n(\Gamma,\Z\Gamma)\otimes R\ \bigoplus\ H^{n+1}(\Gamma,\Z\Gamma)\ast R.
$$
We recall the following facts about cohomological dimension of groups~\cite{Bi},~\cite{Br}: 

(1) $\cd_R\Gamma$ equals the minimal length of projective resolution of $R$ as a trivial $R\Gamma$-module;

(2) For geometrically finite groups $\cd_R\Gamma=\max\{n\mid H^n(\Gamma,R\Gamma)\ne 0\}$.

\begin{prop}\label{ucf}
For any group $\Gamma$ of finite type and any ring $R$, $\cd_R\Gamma\le\cd\Gamma$.
\end{prop}
\begin{proof}
We may assume that $\cd\Gamma<\infty$.
By the UCF if $H^n(\Gamma,R\Gamma)\ne 0$, then either $H^n(\Gamma,\Z\Gamma)\ne 0$ or $H^{n+1}(\Gamma,\Z\Gamma)\ne 0$.
In either case $\cd(\Gamma)\ge n$.
\end{proof}

The main techical result of the paper is the following theorem which will be proven at the end of the section.

\begin{thm}\label{field}
For any geometrically finite group $\Gamma$ there is a field $\k$ such that $\cd\Gamma=\cd_\k\Gamma$.
\end{thm}

{\bf The Kunneth Formula:}~\cite{Bre} {\em For  groups $\Gamma$, $\Gamma'$ of finite type, a $\Z\Gamma$-module $M$, and  a $\Z\Gamma'$ module $M'$ with $M\ast M'=0$ as abelian groups,
there is the equality}
$$
H^n(\Gamma\times\Gamma',M\otimes M')=\bigoplus_k H^{k}(\Gamma,M)\otimes H^{n-k}(\Gamma',M')\oplus \bigoplus_k H^{k+1}(\Gamma,M)\ast H^{n-k}(\Gamma',M').
$$

Usually the Kunneth Formula as well as the Universal Coefficient Formula  are stated as short exact sequences. This  happens because the splitting of those short exact sequences is not natural. Since the naturality is not used in this paper we state both as formulas.

There is a similar formula for $R\Gamma$ and $R\Gamma'$ modules $M$ and $M'$ for any principal ideal domain $R$
in which $\otimes$ and $\ast $ are taken over $R$. When $R=k$ is a field, the Kunneth Formula turns into the following

$$
H^n(\Gamma\times\Gamma',M\otimes_k M')=\bigoplus_{k=0}^n H^{k}(\Gamma,M)\otimes_k H^{n-k}(\Gamma',M').
$$

\begin{prop}\label{log}
For any geometrically finite group $\Gamma$ and any field $\k$,
$$
\cd_\k(\Gamma\times\Gamma)=2\cd_\k\Gamma.$$
\end{prop}
\begin{proof}
Let $\cd_\k\Gamma=n$. Then the vector space $H^n(\Gamma,\k\Gamma)\ne 0$ is nonzero. By the Kunneth Formula over a field,
$$H^{2n}(\Gamma\times\Gamma,\k\Gamma\otimes_\k\k\Gamma)=H^n(\Gamma,\k\Gamma)\otimes_\k H^n(\Gamma,\k\Gamma)\ne 0.$$
\end{proof}
Here is our main result.
\begin{thm}\label{main}
For any geometrically finite group $\Gamma$,
$$
\cd(\Gamma\times\Gamma)=2\cd\Gamma.$$
\end{thm}
\begin{proof}
Let a field $\k$ be as in Theorem~\ref{field}. Then by Proposition~\ref{log} and Proposition~\ref{ucf} $$2\cd\Gamma=2\cd_\k\Gamma=\cd_\k(\Gamma\times\Gamma)\le\cd(\Gamma\times\Gamma).$$
The result follows in view of an obvious inequality $2\cd\Gamma\ge\cd(\Gamma\times\Gamma)$.
\end{proof}
\begin{cor}
For any geometrically finite group $\Gamma$,
$$
\cd(\Gamma^n)=n\cd\Gamma.$$
\end{cor}
\begin{proof}
In view of the obvious inequality $\cd(\Gamma^n)\le n\cd\Gamma$ it suffices to show the inequality $\cd(\Gamma^n)\ge n\cd\Gamma$.

By induction on $k$ we obtain the equality $\cd(\Gamma^{2^k})=2^k\cd\Gamma$. For $n<2^k$ the inequality
$$
2^k\cd\Gamma=\cd(\Gamma^{2^k})\le\cd(\Gamma^n)+\cd(\Gamma^{2^k-n})\le \cd(\Gamma^n)+(2^k-n)\cd\Gamma
$$
implies that $n\cd\Gamma\le\cd(\Gamma^n)$.
\end{proof}
\begin{cor}
Boltyaskii's compactum cannot be a $Z$-boundary of a geometrically finite group.
\end{cor}
\begin{proof}
Suppose that  Boltyanskii's compactum $B$ is a $Z$-boundary of a geometrically finite group $\Gamma$. Then the join product $B\ast B$ is a $Z$-boundary of $\Gamma\times\Gamma$.
In view of Bestvina's formula
$$
\dim_\Z\partial G+1=\cd G.
$$
for a group $G$ with a $Z$-boundary $\partial G$~\cite{Be} we obtain
a contradiction:
$$
5=\dim_\Z(B\ast B)+1=\cd(\Gamma\times\Gamma)=2\cd\Gamma=2(\dim_\Z B+1)=6.
$$
\end{proof}

We use the notations $\Z_p=\Z/p\Z$ for the cyclic group of order $p$, $\Z_{p^\infty}=\lim_{\rightarrow}\{\Z_{p^k}\}$ for the quasi-cyclic group, and $\Z_{(p)}$ for the localization of $\Z$ at $p$.

\begin{prop}\label{prime}
For any geometrically finite group $\Gamma$ there is a prime number $p$ such that $\cd\Gamma=\cd_{\Z_{(p)}}\Gamma$.
\end{prop}
\begin{proof}
Let $\cd\Gamma=n$.  Let $T(A)$ denote the torsion subgroup of $A=H^n(\Gamma,\Z\Gamma)$.
If $A/T(A)\ne 0$, then $A/T(A)\otimes\Z_{(p)}\ne 0$ and, hence, $A\otimes\Z_{(p)}\ne 0$  for all $p$. 
By the UCF, $H^n(\Gamma,\Z_{(p)}\Gamma)\ne 0$. Hence $\cd_{\Z_{(p)}}\Gamma\ge n$.
Proposition~\ref{ucf} implies  that $\cd\Gamma=\cd_{\Z_{(p)}}\Gamma$. 

Let $T(A)=A$. Every torsion abelian group $A$ contains either $\Z_{p^k}$ or $\Z_{p^\infty}$ for some $p$ as a direct summand~\cite{Fu}.
Since  $\Z_{p^k}\otimes\Z_{(p)}\ne 0$ as well as $\Z_{p^\infty}\otimes\Z_{(p)}\ne 0$, we obtain $H^n(\Gamma,\Z\Gamma)\otimes \Z_{(p)}\ne 0$ . Hence $H^n(\Gamma,\Z_{(p)}\Gamma)\ne 0$ and therefore, $\cd\Gamma=\cd_{\Z_{(p)}}\Gamma$. 
\end{proof}

\begin{prop}\label{proper}
Suppose that $J\subset\oplus^mR\Gamma$ is a $R\Gamma$-submodule with the quotient module $M=(\oplus^mR\Gamma)/J$
for which $M\otimes_\Z\Q=0$  where $R=\Z_{(p)}$ for some prime $p$. Then $p^nM=0$ for some $n\in\N$.
\end{prop}
\begin{proof}
We apply induction on $m$. In the case $m=1$ we may assume that $J$
does not contain $e\in\Gamma$. Since $(R\Gamma/J)\otimes\Q=0$ there is $r$ such that $p^re\in J$.
Then $p^rR\Gamma\subset J$. This implies that $p^rM=0$.

Assume that the statement of Proposition holds for $m$. Let $J\subset\oplus^{m+1}R\Gamma$ and $M\otimes\Q=0$ where $M=(\oplus^{m+1}R\Gamma)/J$. 
Let $\pi:\oplus^mR\Gamma\oplus R\Gamma\to\oplus^mR\Gamma$ denote the projection. Let $M'=(\bigoplus_{i=1}^kR\Gamma)/\pi(J)$.
We consider the following commutative diagram 
$$
\begin{CD}
@. 0 @. 0 @. 0 @.  \\
@. @VVV @VVV @VVV @.\\
0 @>>> J\cap R\Gamma @>>> J @>>>\pi(J) @>>>0\\
@. @VVV @VVV @VVV @.\\
0 @>>> R\Gamma @>>> \oplus^{m+1}R\Gamma @>\pi>>\oplus^m R\Gamma @>>> 0\\
@. @VVV @VVV @VVV @.\\
0 @>>> R\Gamma/(J\cap R\Gamma) @>>> M @>\bar\pi>> M' @>>> 0\\
@. @VVV @VVV @VVV @.\\
@. 0 @. 0 @. 0 @.  \\
\end{CD}
$$
with exact columns and rows.

In view of surjectivity of
$
M\otimes\Q\to M'\otimes\Q
$
and the assumption of the Proposition we obtain $M'\otimes\Q=0$.
By induction assumption $p^sM'=0$ for some $s$. 
Tensoring the bottom row with $\Q$  implies that $R\Gamma/(J\cap R\Gamma)\otimes\Q=0$.  By induction assumption
$p^r(R\Gamma/(J\cap R\Gamma))=0$ for some $r$. Then $p^{r+s}M=0$. Indeed, for any $x\in M$ we have
$\bar\pi(p^sx)=p^s\bar\pi(x)=0$ and, hence, $p^sx\in R\Gamma/(J\cap R\Gamma)$. Thus, $p^rp^sx=0$.
\end{proof}

{\em Proof of Theorem~\ref{field}.}
Let $\cd\Gamma=n$. Then $H^n(\Gamma,\Z\Gamma)\ne 0$~\cite{Br}.
Let $p$ be as in Proposition~\ref{prime}. Then $H^n(\Gamma,\Z_{(p)}\Gamma)\ne 0$.
Since $B\Gamma$ is a finite complex, the chain complex for $E\Gamma$ consists of finitely generated free
$\Z\Gamma$ modules:
$$0\to C_n\to C_{n-1}\to\cdots\to C_1\to C_0.$$
Then the cohomology group $H^n(\Gamma,\Z_{(p)}\Gamma)$ is defined by the cochain complex
$$
\begin{CD}
0 @<<< \Hom_{\Z\Gamma}(C_n,\Z_{(p)}\Gamma) @<\delta<<\ Hom_{\Z\Gamma}(C_{n-1},\Z_{(p)}\Gamma) @<<<\cdots.
\end{CD}
$$
Note that $Hom_{\Z\Gamma}(C_n,\Z_{(p)}\Gamma) \cong\oplus^m\Z_{(p)}\Gamma$ for some $m\in\N$. We denote by $F^r=\oplus^r\Z_{(p)}\Gamma$ and
by $J=im(\delta: F^\ell\to F^m)$ the image of $\delta$.
Thus, $F^m/J\ne 0$. If $(F^m/J)\otimes\Q\ne 0$, then $H^n(\Gamma,\Q\Gamma)\ne 0$ and we take $\k=\Q$.
If $(F^m/J)\otimes\Q=0$ we apply Proposition~\ref{proper} to obtain $H^n(\Gamma,\Z_p\Gamma)\ne 0$ for some prime number $p$. Then $\k=\Z_p$.

\section{Cohomological dimension over fields}

We recall the excision theorem for cohomology with coefficients in a sheaf~\cite{Bre}, Proposition~ 12.3.
\begin{thm}
If $A\subset X$ is a closed subset of compact space, then for any sheaf $\mathcal M$ on $X$ there is a natural isomorphism
$$
H^p(X,A;\mathcal M)\cong H^p_c(X\setminus A;\mathcal M).
$$
\end{thm}

The excision theorem and the long exact sequence of pair  imply the following
\begin{prop}\label{skelet}
For a finite complex $X$ and any local coefficient system, $$H^m(X;\mathcal M)=H^m(X^{(m+1)};\mathcal M)$$
for all $m$.
\end{prop}

\begin{thm}
For  geometrically finite groups $\cd_{\Z_{(p)}}\Gamma= \cd_\Q\Gamma $ for all but finitely many prime numbers $p$.
\end{thm}
\begin{proof}
Let $B\Gamma$ be a finite complex of $\dim B\Gamma=n$. Let $m>d=\cd_\Q\Gamma$. Then $H^n(\Gamma,\Q\Gamma)=0$ and
 in view of Proposition~\ref{skelet}, $H^m(B\Gamma^{(m+1)};\Q\Gamma)=0$.
Let $\{C_i,\partial_i\}$ denote the cellular chain complex for $E\Gamma$.
The cohomology groups of $B\Gamma^{(m+1)}$  are defined by means of the  truncated chain complex 
$$0\to C_{m+1}\to C_{m}\to\cdots\to C_1\to C_0.$$

The condition $H^{m}(B\Gamma^{(m+1)},\Q\Gamma)=0$ means that
 the homomorphism 
$$\delta:\Hom_{\Z\Gamma}( C_{m-1},\Q\Gamma)\to \Hom_{\Z\Gamma}(C_m/\partial(C_{m+1}),\Q\Gamma)$$ defined by the boundary homomorphism $\partial:C_{m}\to C_{m-1}$ is onto.
 Let $\{\alpha_1,\dots,\alpha_k\}\in \Hom_{\Z\Gamma}(C_{m-1},\Q\Gamma)$ be a generating set of the free $\Q\Gamma$-module. 
For all $(m-1)$-cells in $B\Gamma$, $\kappa_1,\dots,\kappa_\ell$, we fix their  lifts $\bar\kappa_1,\dots,\bar\kappa_\ell$ in $E\Gamma$.
Let $$\alpha_j(\bar\kappa_i)=\sum \lambda^{i,j}_kg^{i,j}_k\ \ \text{where}\ \ \lambda^{i,j}_k\in\Q\ \ \text{and}\ \ g^{i,j}_k\in\Gamma.$$ Clearly, there is a finite set of primes $\mathcal P_m=\{p_1,\dots,p_{s_m}\}$ such that
$\lambda^{i,j}_k\in\Z[\frac{1}{p_1\cdots p_s}]$. Then the homomorphism $$\delta_p:\Hom_{\Gamma}( C_{m-1},\Z_{(p)}\Gamma)\to \Hom_{\Gamma}(( C_{m}/\partial C_{m+1},\Z_{(p)}\Gamma)$$ is surjective for all $p\notin\mathcal P_m$. Therefore, $H^m(B\Gamma^{(m+1)};\Z_{(p)}\Gamma)=0$
for all  $p\notin\mathcal P_m$. By Proposition~\ref{skelet}, $H^m(B\Gamma;\Z_{(p)}\Gamma)=0$
for all  $p\notin\mathcal P_m$.
 Hence,
$\cd_{\Z_{(p)}}\Gamma\le d$ for all  $p\notin\bigcup_{d<m\le n}\mathcal P_m$.

The inequality $\cd_{\Q}\Gamma\le\cd_{\Z_{(p)}}\Gamma$ completes the proof.
\end{proof}

\begin{prop}\label{k}
For all groups for every prime $p$ and any $k\in\N$,
$$
\cd_{\Z_{p^k}}\Gamma=\cd_{\Z_p}\Gamma.
$$
\end{prop}
\begin{proof}
We apply induction on $k$. The short exact sequence $0\to\Z_{p^k}\to\Z_{p^{k+1}}\to\Z_p\to 0$ produces the short exact sequence of $\Gamma$-modules
$$
0\to \Z_{p^k}\Gamma \to \Z_{p^{k+1}}\Gamma\to\Z_p\Gamma\to 0.
$$
Let $\cd_{\Z_{p^{k+1}}}\Gamma=n$. Then the coefficient long exact sequence 
$$
\dots\to H^n(\Gamma,\Z_{p^k}\Gamma) \to H^n(\Gamma,\Z_{p^{k+1}}\Gamma)\to H^n(\Gamma,\Z_p\Gamma)\to\dots
$$
and the induction assumption imply that $\cd_{\Z_p}\Gamma\ge n$. The same sequence at dimension $n+1$ together with the induction assumption implies that
$\cd_{\Z_p}\Gamma< n+1$.
\end{proof}

\begin{question}
Does the equality $\cd_{Z_p}\Gamma=\cd_{\Z_{(p)}}\Gamma$ for a geometrically finite group $\Gamma$ hold true for all prime numbers $p$ ?
\end{question}

\section{Some remarks}

\subsection{Finitely generated torsion modules over $\Z\Gamma$}
An algebraic reason for the main result of the paper is the fact that for a geometrically finite group $\Gamma$ any finitely generated torsion $\Z\Gamma$-module  does not contain $\Z_{p^\infty}$ for any prime $p$:

\begin{prop}
Let $M$ be a finitely generated $\Z\Gamma$-module for a geometrically finite group $\Gamma$ where the underlying abelian group $M$ is a torsion group. Then for any prime $p$ the subgroup of $p$-torsions
$M[p]\subset M$ is bounded, i.e., there is $r\in\N$ such that $p^rx=0$ for all $x\in M[p]$.
\end{prop}
\begin{proof}
Let $\xi:\oplus^m\Z\Gamma\to M$ be an epimorphism of $\Z\Gamma$-modules. Then the tensor product with $\Z_{(p)}$ defines an epimorphism
$$\xi\otimes 1_{\Z_{(p)}}:(\oplus^m\Z\Gamma)\otimes\Z_{(p)}=\oplus^m\Z_{(p)}\Gamma\to M[p]=M\otimes\Z_{(p)}$$ with the kernel $J=ker(\xi)$. Proposition~\ref{proper} implies that $M[p]$ has bounded order.
\end{proof}

REMARK. The examples of right-angled Coxeter groups $\Gamma$ in~\cite{Dr1} mentioned in the introduction have  geometrically finite subgroups $\Gamma'$ that admit  infinite  finitely generated $\Z\Gamma'$-modules $M$ which are torsions as the additive groups.

\subsection{Application to topological complexity}
Motivated by Topological Robotics, Michael Farber introduced the notion of topological complexity $TC(X)$ of the configuration space $X$~\cite{F} as the minimal number $k$ such that
$X\times X$ can be covered by $k+1$ open sets $U_0,\dots,U_k$ for each of which there is a continuous map $\phi_i:U_i\to C([0,1],X)$ satisfying the condition 
$\phi_i(x,y)(0)=x$ and $\phi_i(x,y)(1)=y$ for all $(x,y)\in U_i$. Since $TC(X)$ is a homotopy invariant, the topological complexity can be defined for groups $\Gamma$ as $TC(B\Gamma)$.
Computation of topological complexity of groups presents a great challenge. The main theorem of this paper allows to complete the computation of topological complexity of hyperbolic groups~\cite{Dr3}
$$
TC(\Gamma)=2\cd\Gamma
$$
originated by Farber and Mescher~\cite{FM}.
Also in the case of geometrically finite groups the main theorem simplifies the formula for topological complexity of the free product of groups~\cite{DS} to the following
$$TC(\Gamma\ast\Gamma)=2\cd\Gamma.$$

\end{document}